\documentclass[11pt]{article}

\usepackage{amsmath}
\usepackage{amsfonts}
\usepackage{amssymb}
\usepackage{amsthm}
\usepackage{amscd}
\usepackage{amsbsy}
\usepackage{graphicx}
\usepackage{indentfirst, latexsym, bm}
\usepackage{bbding}

\usepackage[colorlinks=true]{hyperref}
\advance\textwidth by +1.0in \advance\textheight by +1.0in
\advance\oddsidemargin by -0.5in \advance\evensidemargin by -1.0in
\advance\topmargin by -0.5in
\newtheorem {theorem} {Theorem}[section]
\newtheorem {proposition} [theorem]{Proposition}
\newtheorem {corollary} [theorem]{Corollary}
\newtheorem {lemma}  [theorem]{Lemma}

\newcommand{\R}{{\mathbb R}}
\newcommand{\Z}{{\mathbb Z}}

\def\b{\beta}
\def\d{\delta}

\def\la{\lambda}
\def\g{\gamma}

\title{\large\bf The Decay Rate of Patterson-Sullivan Measures with Potential Functions and Critical Exponents}

\author{
Ziqiang Feng \thanks{College of Mathematics and System Science, Shandong University of Science and Technology, Qingdao, 266590, P.R. China.
e-mail: feng\_math$@$163.com.}
\and
Fei Liu \thanks{College of Mathematics and System Science, Shandong University of Science and Technology, Qingdao, 266590, P.R. China.
e-mail: liufei$@$math.pku.edu.cn.}
\and Fang Wang
\thanks{School of Mathematical Sciences, Capital Normal University, Beijing, 100048, China; and Beijing Center for Mathematics and Information
Interdisciplinary Sciences (BCMIIS), Beijing 100048, P.R. China. e-mail: fangwang@cnu.edu.cn.}  }
\date{\today}

\begin{document}
\maketitle

\begin{abstract}
Basing upon the recent development of the Patterson-Sullivan measures with a H\"older continuous nonzero potential function, we use tools of both dynamics of geodesic flows and geometric properties of negatively curved manifolds to present a new formula illustrating the relation between the exponential decay rate of Patterson-Sullivan measures with a H\"older continuous potential function and the corresponding critical exponent.
\\


\end{abstract}
\section{\bf Introduction}\label{sec1}

\setcounter{section}{1}
\setcounter{equation}{0}\setcounter{theorem}{0}

This article is devoted to the study of the properties of Patterson-Sullivan measures with nonzero potential functions
on the ideal boundary $X(\infty)$ of a simply connected negatively curved Riemannian manifold $X$.

There are various families of measures on $X(\infty)$ indexed by points of $X$
and the members of each family belong to a same measure class. Among these, three kinds of measures,
the Lebesgue measures, the harmonic measures and the Patterson-Sullivan measures, are particularly important.

The topic of this article concerns the Patterson-Sullivan measures, which was first introduced and studied by S. J. Patterson
in the setting of Fuchsian groups (cf.~\cite{Pa}). He constructed a family of absolutely continuous measures
supported on the limit set of the ideal boundary of a Fuchsian group.
Subsequently D. Sullivan extends this construction onto general real hyperbolic spaces (cf.~\cite{Su}).
Then C. Yue (\cite{Yu}) and T. Roblin (\cite{Ro1, Ro}) generalize these results to manifolds of negative curvature.
Recently, Paulin-Pollicott-Schapira (\cite{PPS}) developed a theory of Patterson-Sullivan measures with a nonzero potential function $F$,
and showed that this new kind of measures share many important properties with the classical ones.
Pit-Schapira (\cite{PS}) called it the \emph{Patterson-Sullivan-Gibbs measure}.

The Patterson-Sullivan measures build a connection of the actions of the limit set of a discrete group on the universal covering manifold
with the ergodic theory of the geodesic flow on the quotient manifold, hence play
a significant role in the study of the dynamics of geodesic flows nowadays (cf.~\cite{Ka, Kn1, Kn2, LWW, PS, Ri}).
In \cite{Ka}, V. A. Kaimanovich showed that there exists a natural $1$-to-$1$ correspondence between the set of finite invariant
measures of the geodesic flow on $T^{1}M$ and the set of $\Gamma$-invariant Radon measures on $X^{2}(\infty)$,
and explained that how to construct $\Gamma$-invariant measures on $X^{2}(\infty)$ from the measures on $X(\infty)$ (for example, the Patterson-Sullivan measures and the harmonic measures),
where $M = X/\Gamma$ is a compact quotient manifold and $X^{2}(\infty) = X(\infty)\times X(\infty)\setminus \{(\xi,\xi)| \xi\in X(\infty)\}$.
Furthermore he revealed several properties of these measures.

In this article, we will focus on Patterson-Sullivan-Gibbs measures rather than the classical Patterson-Sullivan measures and generalize
some results of Kaimanovich's.

\section{\bf Basic Concepts and the Main Result}\label{sec2}

\setcounter{section}{2}
\setcounter{equation}{0}\setcounter{theorem}{0}

Let $M$ be a smooth compact negatively curved manifold with pinched sectional curvature $-b^2\leq K \leq -a^2 (b>a>0)$, and $X$ be its Riemannian universal covering manifold. Thus $M = X/\Gamma$ where $\Gamma$ is the fundamental group of $M$. Let $T^1 M$ (resp. $T^1 X$) denote the unit tangent bundle of $M$ (resp. $X$).

For any point $p \in M$ or $X$, and for all $v \in T_{p}M$ or $T_{p}X$,
let $\gamma_{v}$ be the unique geodesic satisfying the initial conditions $\gamma_v(0)=p$ and $\gamma'_v(0)=v$.
In order to simplified the notations, we use $\phi_{t}(v)=\gamma'_v(t)$ to denote the geodesic flow both on $T^{1}M$ and $T^{1}X$.

Two geodesics $\gamma_{1}$ and $\gamma_{2}$ in $X$ are called positively asymptotic (resp. negatively asymptotic), if there exists $C>0$ such that
\[
d(\gamma_{1}(t),\gamma_{2}(t))\leq C,~\forall t \geq 0 ~~(resp.  ~\forall t \leq 0).
\]

Here $d$ is the distance function induced by the Riemannian metric. It is easy to see that the positive asymptoticity (resp. negatively asymptoticity) establishes an equivalence relation on the set of all the geodesics on $X$.
Given a geodesic $\gamma$, we use $\gamma(+\infty)$ (resp. $\gamma(-\infty)$) to denote the equivalence class of geodesics positively asymptotic (resp. negatively asymptotic) to $\gamma$. Also we can consider this equivalence class as a point at infinity. The set of all points at infinity is usually denoted by $X(\infty)$.

In \cite{PPS}, Paulin-Pollicott-Schapira has made a comprehensive and systematic study of Patterson-Sullivan measures with nonzero potential function $F$. The following notations and results are cited from $\S3.1-\S3.6$ of \cite{PPS}.

Let $F:T^{1}M \rightarrow \mathbb{R}$ be a H\"older continuous function and $\widetilde{F}:T^{1}X \rightarrow \mathbb{R}$ be the lift function of $F$, thus $\widetilde{F}$ is a H\"older continuous, $\Gamma$-invariant function, called a \emph{potential}.

For any $x$, $y\in X$, let
\[\int_{x}^{y} \widetilde{F} = \int_{0}^{d(x,y)} \widetilde{F}(\phi_{t}(v))dt,
\]
where $v\in T_{x}^{1}X$ such that $\pi(\phi_{d(x,y)}(v))=y$. Here $\pi : T^{1}X \rightarrow X$ is the standard projection map and $\phi_{t} : T^{1}X\rightarrow T^{1}X$ is the geodesic flow.

Fix $x,y\in X$, the Poincare series of $(\Gamma,F)$ is the map
\begin{eqnarray*}
Q_{\Gamma,F,x,y} :&\R& \rightarrow [0,+\infty],\\
& s & \mapsto Q_{\Gamma,F,x,y}(s) = \sum_{\gamma \in \Gamma}\mathrm{e}^{\int_{x}^{\gamma y}(\widetilde{F}-s)}.
\end{eqnarray*}

We define the critical exponent $\delta_{\Gamma,F}$ of $(\Gamma ,F)$ by
\[
	\delta_{\Gamma,F} = \limsup_{n\rightarrow +\infty} \frac{1}{n} \log \sum_{\substack{\gamma\in\Gamma \\ n-1<d(x,\gamma y)\leq n}} \mathrm{e}^{\int_{x}^{\gamma y} \widetilde{F}}\quad\in [-\infty,+\infty].
\]

In \cite{PPS}, Paulin-Pollicott-Schapira proved the following result.

\begin{theorem}[cf.~Paulin-Pollicott-Schapira \cite{PPS}]\label{them1}
If $\delta_{\Gamma,F}<\infty$, then there exists a family of finite nonzero (positive Borel) measures $\{\mu^{F}_x\}_{x\in X}$ on $X(\infty)$, such that, for any $\gamma \in \Gamma$, for any $x,y\in X$, and for each $\xi\in X(\infty)$, we have
\begin{eqnarray*}
\gamma_{\ast}\mu^{F}_{x} &=& \mu^{F}_{\gamma x},\\
\frac{d\mu^{F}_{x}}{d\mu^{F}_y}(\xi) &=& e^{-C_{F-\delta_{\Gamma,F},\xi}(x,y)},
\end{eqnarray*}
where $C_{F-\delta_{\Gamma,F},\xi}(x,y) = \lim\limits_{t\rightarrow +\infty}\left\{ \int_{y}^{\xi_{t}}(\widetilde{F}-\delta_{\Gamma,F})-
\int_{x}^{\xi_{t}}(\widetilde{F}-\delta_{\Gamma,F})\right\}$ is the Gibbs cocycle for the potential function $F$, here $t\mapsto \xi_{t}$ is any geodesic ray ending at $\xi \in X(\infty)$.
\end{theorem}

$\{\mu^{F}_{x}\}_{x\in X}$ are called \emph{Patterson-Sullivan(-Gibbs) measures of dimension $\delta_{\Gamma,F}$}. When $F \equiv 0$, this definition coincides with the classical Patterson-Sullivan measures. 
By the definition, we know that all $\mu^{F}_x (x\in X)$ belong to the same measure class.

Let $\lambda$ be the Liouville measure on the unit tangent bundle $T^{1}X$. Since the geodesic flow $\phi_t: T^{1}M \rightarrow T^{1}M$
is ergodic with respect to the Liouville measure on $T^{1}M$ (see \cite{An}), by Birkhorff ergodic theorem, for $\lambda$-a.e. $v\in T^{1}X$, the following limit
\[
\lim\limits_{t\rightarrow +\infty} \frac{1}{t} \int_{0}^{t} \widetilde{F}(\gamma_v'(s))ds
\]
exists and is independent of $v$, it is in fact the integral of $F$ with respect to the Liouville measure on $T^{1}M$.
We denote this limit by $\lambda_F$, i.e., for $\lambda$-a.e. $v\in T^{1}X$,
\[
\lambda_F = \lim\limits_{t\rightarrow +\infty} \frac{1}{t} \int_{0}^{t} \widetilde{F}(\gamma_v'(s))ds.
\]

The following theorem is the main result of this article. A special case when $F\equiv 0$ has been proved by Kaimanovich in \cite{Ka}.

\begin{theorem}\label{them2}
Let $M,X,\widetilde{F},F$ and $\Gamma$ are the ones as mentioned above. Let $\{\mu^{F}_{x}\}_{x\in X}$ be the Patterson-Sullivan measures constructed in $Theorem~\ref{them1}$, then for any $x\in X$ and $\lambda$-a.e. $v\in T^{1}X$,
\[
\lim\limits_{t\rightarrow +\infty} \frac{1}{t} \log \mu^{F}_x (B_{x,t}(\gamma_v(+\infty)))=-\delta_{\Gamma,F}+\lambda_F
\]
\end{theorem}
Here $B_{x,t}(\gamma_v(+\infty))\subseteq X(\infty)$ is a neighborhood of $\gamma_v(+\infty)$ in $X(\infty)$ and its detailed definition will be given in the next section.

Obviously the limit value $-\delta_{\Gamma,F}+\lambda_F\leq 0$ since the measure $\mu^F_x$ is finite. A straightforward corollary follows from this observation.

\begin{corollary}
Under the same condition in Theorem \ref{them2}, we have the inequality $$\lambda_F\leq\delta_{\Gamma,F}.$$
\end{corollary}
We must point out that this corollary present the relation between the critical exponent of $(\Gamma, F)$ and the average of the H\"older continuous potential $F$.
In the case $F\equiv 0$, this is straightforward. However, when $F$ is a non-zero potential function, this relation is highly nontrivial in general.

Although it is in the early stage of development, we are optimistic about broad application prospects of this theorem on both dynamics and geometry, and outline our future research regarding related topics. Our plan for next paper will be on computing the Hausdorff dimension of $\mu^{F}_{x}$ (see for example \cite{Ka}).

In the next section we will give the proof of Theorem $\ref{them2}$. Unlike the traditional method based on the theory of hyperbolic groups like in \cite{Ka}, we introduce the dynamical properties of negatively curved manifolds to build this formula.
\section{\bf Proof of Theorem~$\ref{them2}$ }

\setcounter{section}{3}
\setcounter{equation}{0}\setcounter{theorem}{0}

Let $\mbox{Isom}(X)$ be the group of all isometry transformations of $X$. We call an isometry $\alpha\in \mbox{Isom}(X)$ an \emph{axial element} if there exists a geodesic $\gamma$ in $X$ and a $T>0$ such that for any $t\in\R$, $~\alpha(\gamma(t))=\gamma(t+T)$. Correspondingly $\gamma$ is called the \emph{axis} of $\alpha$.
In fact when $M = X/\Gamma$ is compact, every element of $\Gamma$ is an axial element. The following result is useful to us.

\begin{proposition}[cf.~Ballmann \cite{Ba}]\label{pro}
Let $X$ be a simply connected manifold with pinched negative curvature and $X/\Gamma$ is compact, if
$\gamma$ is an axis of $\alpha \in \Gamma \subset \mbox{Isom}(X)$, then for any neighborhoods $U\subset X(\infty)$ of $\gamma(-\infty)$
and $V\subset X(\infty)$ of $\gamma(+\infty)$, there exists $N\in \Z^{+}$ such that
\[ \alpha^{n}(X(\infty)-U)\subset V,~~~~\alpha^{-n}(X(\infty)-V)\subset U, ~\forall~n\geq N.
\]
\end{proposition}

In fact Ballmann proved this result for a broader class of manifolds known as rank $1$ manifolds of non-positive curvature.
Later Watkins (\cite{Wa}) and Liu-Wang-Wu (\cite{LWW}) extended his result onto the rank $1$ manifolds without focal points.
Refer to \cite{LLW,LW,LWW,LWW1,LWW2,LZ} for more information about recent works of the geometries and dynamics of geodesic flows
on manifolds without focal/conjugate points.

In general, $X(\infty)$ can be identified with the $(dim X-1)$-sphere $T_{x}^{1}X$ at any point $x\in X$.
For any geodesic $\gamma$ in $X$ and for any point $x\in X$, we can find a unique geodesic $\beta$ in $X$ starting from $x$ and is
positive asymptotic to $\gamma$, i.e. $\beta(0) = x$ and $\beta(+\infty) = \gamma(+\infty)$. For more details, see \cite{EO}.
Thus for any $\xi\in X(\infty)$ and $x\in X$, we use $\gamma_{x,\xi}$ to denote the geodesic connecting $x$ and $\xi$. That is to say
$\gamma_{x,\xi}(0)= x$ and $\gamma_{x,\xi}(+\infty)=\xi$.
Furthermore, for each $x\in X$ and $\xi,\eta\in X(\infty)$, we can define the angle between $\xi$ and $\eta$ seen from $x$ by
\[\angle_x(\xi,\eta) ~:=~ \angle_x(\gamma_{x,\xi}'(0),\gamma_{x,\eta}'(0)).
\]

Now for each $x\in X$ and $\xi \in X(\infty)$, $\forall a>0$, define
\[A_{x,a}(\xi)~=~\{\eta \in X(\infty)~|~\angle_x(\xi,\eta)<a\}\]
to be the cone neighborhood of $\xi$ in $X(\infty)$.
For each $x\in X$, define a distance function $d_x$ on $X(\infty)=\{(\xi,\eta)\in X(\infty)\times X(\infty) | \xi\neq\eta\}$ by
\[d_x(\xi,\eta)=t~\Leftrightarrow~d(\gamma_{x,\xi}(t),\gamma_{x,\eta}(t))=1.\]
Under this distance $d_x,\forall t\geq \frac{1}{2}$, we define a neighborhood of $\xi\in X(\infty)$ by
\[B_{x,t}(\xi)~=~\{\eta \in X(\infty)~|~d_x(\xi,\eta)>t\}.\]

\begin{lemma}\label{lem1}
There exists $C>0$ such that for any $x\in X$ and for any $\xi \in X(\infty)$, we have $\mu^{F}_x(A_{x,\frac{\pi}{2}}(\xi))>C$.
\end{lemma}

\begin{proof}
We know the Patterson-Sullivan measure $\{\mu^{F}_x\}_{x\in X}$ and the Gibbs cocycle $C_{F-\delta_{\Gamma,F},\cdot}(\cdot,\cdot)$ are $\Gamma$-invariant, and $C_{F-\delta_{\Gamma,F},\xi}(\cdot,\cdot)$ is continuous (cf.~\cite{PPS}). In fact given $M=X/\Gamma$ is compact, we only need to prove Lemma~$\ref{lem1}$ for one fixed point $x_0\in X$.

As we have mentioned in Section \ref{sec2}, $X(\infty)$ is homeomorphic to a $(\dim X-1)$-sphere, thus compact. Therefore there exist finite many points $\{\xi_i\}_{i\in I}\in X(\infty)$ and an angle $\theta>0$ such that $\forall \xi \in X(\infty)$, there exists at least one $\xi_i$ satisfying
\begin{eqnarray}\label{e1}
 A_{x_0,\theta}(\xi_i)\subseteq A_{x_0,\frac{\pi}{2}}(\xi).
\end{eqnarray}
Since the axes are dense in the space of geodesics in X (cf.~\cite{Ba,Kn1,Kn2}), by changing $\theta$ smaller, we can choose each point in $\{\xi _i\}_{i\in I}$ to be an endpoint of some axis of an axial element $\alpha_i\in \mbox{Isom}(X)$.

The compactness of $M=X/\Gamma$ also implies that the total mass of the Patterson-Sullivan measure $\{\mu^{F}_x\}_{x\in X}$ is uniformly bounded both from above and below. Since for any $x\in X$, $\mathrm{Supp}\mu_x=X(\infty)$ (cf.~\cite{Kn1, LWW}),
we can choose two separated open subsets $S_1$ and $S_2$ of $X(\infty)$ such that
\[\mu^{F}_{x_0}(S_i)>0,  ~i=1,2.
\]

For each $\xi_i(i\in I)$, as $S_1$ and $S_2$ are separated, we know either $\xi_i\notin S_1$ or $\xi_i\notin S_2$.
Thus by Proposition $\ref{pro}$, there exists $n_i\in\Z$ such that either $\alpha_{i}^{n_i}S_1$ or $\alpha_{i}^{n_i}S_2$ is contained in $A_{x_0,\theta}(\xi_i),$

Combining (\ref{e1}), we can conclude at least one of the following relations holds
\begin{eqnarray*}
\alpha_{i}^{n_i}S_1~\subseteq~A_{x_0,\theta}(\xi_i)~\subseteq~A_{x_0,\frac{\pi}{2}}(\xi),\\
\alpha_{i}^{n_i}S_2~\subseteq~A_{x_0,\theta}(\xi_i)~\subseteq~A_{x_0,\frac{\pi}{2}}(\xi).
\end{eqnarray*}
Thus in order to prove this lemma, by the finiteness of the set $\{\xi_i\}_i\in I\subseteq X(\infty)$, we only need to show that
\[\mu^{F}_{x_0}(\alpha_{i}^{n_i}S_1)>0,~~\mu^{F}_{x_0}(\alpha_{i}^{n_i}S_2)>0,~~i\in I.
\]

Without loss of generality, we will show $\mu^{F}_{x_0}(\alpha_{i}^{n_i}S_1)>0$.

By the definition, $\mu^{F}_{x_0}(\alpha_{i}^{n_i}S_1)=(\alpha_i^{-n_i} \mu^{F}_{x_0})(S_1)$.
Theorem~$\ref{them1}$ implies that
\[ \mu^{F}_{\alpha_i^{-n_i}x_0}(S_1)=\mu^{F}_{x_0}(\alpha_{i}^{n_i}S_1),
\]
and
\[\frac{d\mu^{F}_{\alpha_i^{-n_i}x_0}}{d\mu^{F}_{x_0}}(\xi)=e^{-C_{F-\delta_{\Gamma,F},\xi}(\alpha_i^{-n_i}x_0,x_0)}.
\]

By the facts that $C_{F-\delta_{\Gamma,F},\xi}(\alpha_i^{-n_i}x_0,x_0)$
is continuous with respect to $\xi\in X(\infty)$ and $X(\infty)$ is compact, we know the function
\[C_{F-\delta_{\Gamma,F,\cdot}}(\alpha_i^{-n_i}x_0,x_0):~~X(\infty)~\rightarrow~\R
\]
is uniformly bounded from both above and below, thus $\mu_{x_0}(S_1) >0$ implies that
\[\mu^{F}_{\alpha_i^{-n_i}x_0}(S_1)>0,\]
therefore
\[\mu^{F}_{x_0}(\alpha_{i}^{n_i}S_1)=\mu^{F}_{\alpha_i^{-n_i}x_0}(S_1)>0.\]

We complete the proof of this lemma.
\end{proof}

\begin{lemma}\label{l2}
For any $x\in X$ and $\xi\in X(\infty)$, let $\g:=\g_{x,\xi}$ denote the geodesic ray connecting $x$ and $\xi$.
There exist positive constants $N$ and $K$ depending only on the curvature bounds, satisfying
\[A_{\g(t+K),\frac{\pi}{2}}(\xi)\subseteq B_{x,t}(\xi),\ \ \ \forall t\geq N.
\]
\end{lemma}
\begin{proof}
We will prove this lemma by contradiction. Assume this Lemma fails, then there exists $x\in X, ~\xi\in X(\infty)$, and sequences $\left\{t_{i}\right\}_{i=1}^{+\infty}\subseteq\R, \ \left\{s_{i,j}\right\}_{i,j=1}^{+\infty}$ with $\lim\limits_{i\rightarrow+\infty}t_i=+\infty,\ \lim\limits_{j\rightarrow+\infty}s_{i,j}=+\infty$ such that for any $i$, $j$
\begin{equation}\label{0}
A_{\g(t_i+s_{i,j}),\frac{\pi}{2}}(\xi)\not\subseteq B_{x,t_i}(\xi).
\end{equation}

We will need the following facts.
\\\textbf{Facts}. Let $\b$ be a geodesic in $X$. $\forall t\in \R$, let $D_t$ be the hyper-surface that orthogonal to $\b$ at $\b(t)$. We denote the points at infinity of this hyper-surface by $D_t(\infty)$, then we have
\begin{equation}\label{eq:squareum1}
D_{t_1}(\infty)\cap D_{t_2}(\infty)=\emptyset,\ \ D_{t_1}\cap D_{t_2}=\emptyset,\ \ \forall t_1\neq t_2,
\end{equation}
\begin{equation}\label{2}
\bigcup\limits_{t\in \R}D_t(\infty)= X(\infty)-\{\b(-\infty),\b(+\infty)\},
\end{equation}
\begin{equation}\label{3}
A_{\b(t),\frac{\pi}{2}}(\xi)-\{\xi\}=\bigcup\limits_{s\geq t}D_s(\infty).
\end{equation}

We will only prove (\ref{eq:squareum1}) here. (\ref{2}) and (\ref{3}) are straightforward corollaries of (\ref{eq:squareum1}).

Suppose $D_{t_1}(\infty)\cap D_{t_2}(\infty)\neq\emptyset$, take $\eta\in D_{t_1}(\infty)\cap D_{t_2}(\infty)$,
then by the definitions of $D_t$ and $D_t(\infty)$, the sum of interior angles of the geodesic triangle
$\bigtriangleup_{\b(t_1)\b(t_2)\eta}$ is greater than $\frac{\pi}{2}+\frac{\pi}{2}=\pi$,
which is not possible because $X$ is a negatively curved manifold with the curvature $K\leq-a^2<0$. Thus $D_{t_1}(\infty)\cap D_{t_2}(\infty)=\emptyset$.
Similarly we can show that $D_{t_1}\cap D_{t_2}=\emptyset$.

Now return to the proof of the lemma. If (\ref{0}) holds, we can fix $i\in \Z$ and take
\[
\xi_{i,j}\in A_{\b(t_i+s_{i,j}),\frac{\pi}{2}}(\xi)-B_{x,t_i}(\xi),
\]

By (\ref{3}), there exists $t_{i,j}>t_i +s_{i,j}$, such that
\[
\xi_{i,j}\in D_{t_{i,j}}(\infty).
\]
We can choose $\{t_{i,j}\}_{j=1}^{+\infty}$ to be an increasing sequence with respect to the index $j$. (\ref{eq:squareum1}), (\ref{2}) and (\ref{3}) imply that $\lim\limits_{j\rightarrow+\infty}t_{i,j}=+\infty$. Thus for the fixed $i\in \Z$, we obtain a sequence of points at infinity
\begin{equation}\label{4}
\{\xi_{i,j}\}_{j=1}^{+\infty}\subseteq X(\infty)-B_{x,t_i}(\xi).
\end{equation}
By the compactness of $X(\infty)$, we can assume that
\[
\overline{\xi}_{i}=\lim\limits_{j\rightarrow+\infty}\xi_{i,j}.
\]
The fact $\xi \in B_{x,t_i}(\xi)$ and (\ref{4}) imply that $\overline{\xi}_{i}\neq\xi$, then (\ref{3}) shows that there exists $t>t_i$, such that
\begin{equation}\label{5}
\overline{\xi}_{i}\in D_t(\infty).
\end{equation}
By (\ref{3}), for a fixed $i\in \Z$, the sets $\{A_{\b(t_i+s_{i,j}),\frac{\pi}{2}}\}$ are nested, thus (\ref{5}) is not possible. Therefore (\ref{0}) is not true. We complete the proof of this lemma.
\end{proof}

Finally we will prove Theorem~$\ref{them2}$.
\begin{proof}
By Birkhorff ergodic theorem, for $\la$-a.e.~$\ v\in T^1X$, the limit
\[
\lambda_F=\lim\limits_{t\rightarrow+\infty}\frac{1}{t}\int_0^t\widetilde{F}(\g_v'(s))ds
\]
exists and is independent of $v$. For any $v\in T^1X$, let $y=\g_v(t),\ t>0$.
Given $X/\Gamma$ is compact and the Gibbs cocycle $C_{F-\d_{\Gamma,F,\xi}}(x,y)$ is continuous and $\Gamma$-invariant(see \cite{PPS}),
we know there exists a constant $D$ depending only on the curvature bounds and the function $F$, such that
\[
\big|C_{F-\d_{\Gamma,F},\g_v(+\infty)}(x,y)-C_{F-\d_{\Gamma,F},\xi}(x,y)\big|\leq D,\ \forall\xi\in B_{x,t}(\g_v(+\infty)).
\]
By the definition,
\[
C_{F-\d_{\Gamma,F},\g_v(+\infty)}(x,y)=-\int_0^t\widetilde{F}(\g_v'(s))ds+t\cdot\d_{\Gamma,F}.
\]
thus
\begin{eqnarray*}
	\int\limits_{B_{x,t}(\g_v(+\infty))}\mathrm{e}^{-D}d\mu^{F}_y(\xi) &\leq& \mathrm{e}^{t\cdot\d_{\Gamma,F}-\int_0^t\widetilde{F}(\g_v'(s))ds}\int\limits_{B_{x,t}(\g_v(+\infty))}\mathrm{e}^{-C_{F-\d_{\Gamma,F},\xi}(x,y)}d\mu^{F}_y(\xi)\\
	&\leq& \int\limits_{B_{x,t}(\g_v(+\infty))}\mathrm{e}^Dd\mu^{F}_y(\xi),
\end{eqnarray*}
i.e.,
\begin{eqnarray*}
	\mathrm{e}^{-D}\cdot\mu^{F}_y(B_{x,t}(\g_v(+\infty))) &\leq& \mathrm{e}^{t\cdot\d_{\Gamma,F}-\int_0^t\widetilde{F}(\g_v'(s))ds}\cdot\mu^{F}_x(B_{x,t}(\g_v(+\infty)))\\ &\leq& \mathrm{e}^D\cdot\mu^{F}_y(B_{x,t}(\g_v(+\infty)))
\end{eqnarray*}
Therefore
\begin{equation}\label{6}
\Big|t\cdot\d_{\Gamma,F}-\int_0^t\widetilde{F}(\g_v'(s))ds+\ln{\mu^{F}_x(B_{x,t}(\g_v(+\infty)))}-\ln{\mu^{F}_y(B_{x,t}(\g_v(+\infty)))}\Big|\leq D
\end{equation}
By Lemma $\ref{l2}$, we know there exist positive constants $N$ and $K$ depending only on the curvature bounds such that
\begin{equation}\label{7}
A_{\g_v(t+K),\frac{\pi}{2}}(\g_v(+\infty))\subseteq B_{x,t}(\g_v(+\infty)),\ \ \forall t\geq N.
\end{equation}
Lemma~$\ref{lem1}$ implies that there exists a constant $C$ which is independent of $\g_v(t+K)$, such that
\begin{equation}\label{8}
\mu^{F}_{\g_v(t+K)}(A_{\g_v(t+K),\frac{\pi}{2}}(\g_v(+\infty)))>C.
\end{equation}
(\ref{6}), (\ref{7}) and  (\ref{8}) imply that
\[
\lim\limits_{t\rightarrow+\infty}\frac{1}{t}\ln\mu^{F}_x(B_{x,t}(\g_v(+\infty)))
=-\d_{\Gamma,F}+\la_F.
\]

This completes the proof of Theorem $\ref{them2}$.
\end{proof}

\section*{\textbf{Acknowledgements}}
F. Liu is partially supported by NSFC under Grant Nos. 11571207 and 11871045.
F.~Wang is partially supported by NSFC under Grant No.11871045 and by the State Scholarship Fund from China Scholarship Council (CSC).

\end{document}